\documentclass[3p,times]{article}

\usepackage{graphicx, caption, subcaption}
\usepackage{amssymb}
\usepackage{amsmath}
\usepackage{amsthm}
\usepackage[final]{showkeys}
\usepackage{microtype}
\usepackage{color}
\usepackage{graphicx}
\usepackage[figuresright]{rotating}
\usepackage[hmargin=3cm,vmargin={3.5cm,4cm}]{geometry}
\usepackage{bm}

\newtheorem{te}{Theorem}[section]
\newtheorem{example}{Example}

\newtheorem{os}[te]{Remark}
\newtheorem{prop}[te]{Proposition}

\newtheorem{coro}[te]{Corollary}
\newtheorem{Obs}{Observation}

\numberwithin{equation}{section}
\title{Hermite, Higher order Hermite, Laguerre type polynomials and Burgers like equations.}

\author{Giuseppe Dattoli, Roberto Garra, Silvia Licciardi\\	
	ENEA - Frascati Research Center, Via Enrico Fermi 45, 00044, \\Frascati, Rome, Italy  \\
	Section of Mathematics, Università Telematica Internazionale, Italy\\
	University of Palermo,  Department of Engeneering, \\
	viale delle Scienze, Palermo, Italy, orcid  0000-0003-4564-8866
}

\date{\today}
\begin{document}
	
	\maketitle

	\abstract{The multivariable version of ordinary and generalized Hermite  polynomials are the natural solutions of the classical heat equation and of its higher order versions. We derive the associated Burgers equations and show that analogous non-linear partial differential equations can be derived for Laguerre polynomials and for the relevant generalizations.}\\
	

\section{Introduction and preliminaries}

The theory of special functions and polynomials has been the subject of intensive programs of research during the years. Most of these studies have been directed towards  the framing of the relevant theory within a unifying context. Which from time to time has been identified with the theory of differential equations \cite{sea}, the group theoretic point of view \cite{talm}, umbral calculus \cite{Roman} and symbolic formalism \cite{pino}, just to quote a few of them.
An important output of these studies has been the formalization of the theoretical foundations of the relevant properties in terms of differential equations, be they ordinary or partial. Multi-variable special functions are the natural solutions of well- known equations  of mathematical physics and the possibility of extending these equations to non-linear forms is the subject of the present study.\\

The two variable Hermite polynomials are defined by \cite{dattoli}
\begin{equation}\label{H}
	H_n(x,y) = n!\sum_{r=0}^{\lfloor\frac{n}{2}\rfloor}\frac{y^r x^{n-2r}}{(n-2r)!r!}.
\end{equation}
and satisfy the parabolic equation
\begin{equation}\label{diff}
	\frac{\partial Z}{\partial y} = \frac{\partial^2 Z}{\partial x^2} 
\end{equation}
under the condition $Z(x,0) = x^n$.
Albeit two variable Hermite are well known in the mathematical literature and trace back to Hermite himself \cite{herm}, Appèl and Kampè dè Fèrièt \cite{app}, they received a wider diffusion when they have been exploited as auxiliary solutions of the heat equation, under the name of heat polynomials, for more specific details see Ref. \cite{wid}.
A further element of interest is due to the fact that they can be associated with non- linear differential equations realizing the so called Burgers hierarchy \cite{cao,levi,kud}.
Here we show that it is possible to construct particular interesting solutions for a wide class of nonlinear partial differential equations by using Hermite and Laguerre polynomials and Hopf-Cole transformation (HCt). We introduce the problem through the following example.

\begin{example}
 We recall that, by applying\footnote{We have omitted the arguments in Eq. \eqref{ch} for conciseness.} the HCt \cite{kud}, we obtain
\begin{equation}\label{ch}
	u(x,y) = \displaystyle\frac{\partial Z}{\partial x}\bigg/Z= \frac{\partial}{\partial x}\ln Z\;,
\end{equation}
with $Z(x,y)$ given by \eqref{H} satisfying Eq. \eqref{diff}. We can easily show that the function $u(x,y)$ satisfies a non-linear equation
presenting simultaneously non-linear advective and diffusive terms (the so called Burgers equation). Keeping the partial derivative of with respect to the y-variable, we obtain
\begin{equation}
	\frac{\partial u}{\partial y} = \frac{\displaystyle \left(\frac{\partial^2 Z}{\partial y \partial x}\right)Z-\frac{\partial Z}{\partial y}\frac{\partial Z}{\partial x}}{Z^2}.
\end{equation}
The use of Eq.\eqref{diff} yields
\begin{equation}
	\frac{\partial u}{\partial y} = \frac{\displaystyle \left(\frac{\partial^3 Z}{\partial x^3}\right)Z-\frac{\partial^2 Z}{\partial x^2}\frac{\partial Z}{\partial x}}{Z^2}.
\end{equation}
Then, according to Eq. \eqref{ch}, we find that the function $u$ satisfies the equation
\begin{equation}\label{bu}
		\frac{\partial u}{\partial y} = 	\frac{\partial^2 u}{\partial x^2}+2u	\frac{\partial u}{\partial x}.
\end{equation}
This means that the Burgers equation \eqref{bu} admits a solution of the form (see also \cite{levi})
\begin{equation}
u(x,y) = \frac{nH_{n-1}(x,y)}{H_n(x,y)}.
\end{equation} 
\end{example}
This result is a particular case, to be framed within a more general context. 
For example, the third order heat equation 
\begin{equation}\label{hdif}
	\frac{\partial \Psi}{\partial t} = \frac{\partial^3 \Psi}{\partial x^3},
\end{equation}
is associated with a non-linear equation, generalizing the already mentioned Burgers. The use of the same transformation given in Eq. \eqref{ch} and extension of the procedure outlined afterwards allows  to prove that the function $u$ satisfies the non-linear form \cite{levi}
\begin{equation}\label{thb}
	\frac{\partial u}{\partial y} = \frac{\partial u^3}{\partial x} +3\left(\frac{\partial u}{\partial x} \right)^2+3u\frac{\partial^2 u}{\partial x^2} +\frac{\partial^3 u}{\partial x^3}
\end{equation}
which, in a more compact form, reads
\begin{equation}\label{n3}
\frac{\partial u}{\partial y} = \frac{\partial}{\partial x}\left( \frac{\partial}{\partial x}+u\right)^2 u .
\end{equation}

\begin{Obs}
The last equations are obtained by proceeding on the basis of the paradigm reported below:
\begin{itemize}
	\item [i)] introduce the auxiliary function  defined as in Eq.\eqref{ch}, namely
	\begin{equation}\label{ch2}
	u(x,y) = \displaystyle\frac{\partial \Psi}{\partial x}\bigg/\Psi ,
\end{equation}
	\item [ii)] keep the derivative of both sides of Eq. \eqref{ch2} with respect to  and obtain
\begin{equation}\label{10}
\frac{\partial u}{\partial y} = \frac{\displaystyle \left(\frac{\partial^2 \Psi}{\partial y \partial x}\right)\Psi-\frac{\partial \Psi}{\partial y}\frac{\partial \Psi}{\partial x}}{\Psi^2} = \frac{\displaystyle \left(\frac{\partial^4 \Psi}{\partial x^4}\right)\Psi-\frac{\partial^3 \Psi}{\partial x^3}\frac{\partial \Psi}{\partial x}}{\Psi^2} ,
\end{equation}
\item [iii)]use the identity 
\begin{equation}
\frac{\partial \Psi}{\partial x} = u \Psi	.
\end{equation}
In Eq. \eqref{10}, eventually ending up with
\begin{equation}\label{lab12}
\begin{split}
	\frac{\partial u}{\partial y} &= \frac{\displaystyle \frac{\partial^3 (u\Psi)}{\partial x^3}-u\frac{\partial^3 \Psi}{\partial x^3}}{\Psi}=\frac{\displaystyle \Psi \frac{\partial^3 u}{\partial x^3}+3\frac{\partial \Psi}{\partial x}\frac{\partial^2 u}{\partial x^2}+3\frac{\partial u}{\partial x}\frac{\partial^2 \Psi}{\partial x^2}}{\Psi} \\
& = \frac{\partial^3 u}{\partial x^3}+3u\frac{\partial^2 u}{\partial x^2}+\frac{3}{\Psi}\frac{\partial u}{\partial x}\frac{\partial (u\Psi)}{\partial x}
\end{split}
\end{equation}
Expanding the last term on the rhs of the last line in \eqref{lab12}, we eventually end up with Eq. \eqref{thb}.
\end{itemize}
\end{Obs}

As is well known the third order Hermite polynomials \cite{dattoli}
\begin{equation}\label{H3}
	H_n^{(3)}(x,y)  = n!\sum_{r=0}^{\lfloor\frac{n}{3}\rfloor}\frac{y^r x^{n-3r}}{(n-3r)!r!} 
\end{equation}
satisfies the equation \eqref{hdif} and 
\begin{equation}
	\frac{\partial H_n^{(3)}(x,y)}{\partial x} = n H_{n-1}^{(3)}(x,y),
\end{equation}
therefore, we have that a particular solution for the nonlinear equation \eqref{n3} is given by 
\begin{equation}
	u(x,y) = \frac{n H_{n-1}^{(3)}(x,y)}{H_n^{(3)}(x,y)}.
\end{equation}
This scheme can be used to construct exact particular solutions for higher order nonlinear PDEs. \\

The extension of the method to higher order “heat” equations is fairly straightforward. We find indeed that the higher order Burgers-type equations
\begin{equation}\label{buu}
	\frac{\partial u}{\partial y}  = \frac{\partial}{\partial x}\left( \frac{\partial}{\partial x}+u\right)^{\!m-1} u
\end{equation}
can be linearized by means of the transform \eqref{ch2} to the higher order diffusive equation

\begin{equation}
	\frac{\partial \Psi}{\partial t} = \frac{\partial^m\Psi}{\partial x^m}, \qquad m>3.
\end{equation}
It is well-known that higher order diffusive equations of this type
admits polynomial solutions of the form
\begin{equation}
	H_n^{(m)}(x,y)  = n!\sum_{r=0}^{\lfloor\frac{n}{m}\rfloor}\frac{y^r x^{n-mr}}{(n-mr)!r!}.
\end{equation}
Therefore, by using the HCt, it is possible to solve higher order Burger-type equations \eqref{buu} in a similar way.\\

In this article we pursue an investigation started in Ref. \cite{logistic}, where the authors have studied extensions of the Burgers equations including non-standard derivative forms. Our intention is that of exploring this subject in a wider context, by an appropriate discussion involving not only Hermite but also Laguerre and other forms of hybrid polynomials.\\

To conclude this section, we observe another interesting application of this method in order to obtain solutions based on Hermite polynomials for a more general class of Burgers equations.

\begin{example}
In the recent paper \cite{vag}, the author considered the following Burgers-type equation with a variable coefficient
\begin{equation}
	\frac{\partial u}{\partial t}+	\frac{2}{F(y,t)}u \frac{\partial u}{\partial x} = 	\frac{\partial^2 u}{\partial x^2}+	\frac{\partial^2 u}{\partial y^2}.
\end{equation}
Considering the following Hopf-Cole transformation
\begin{equation}
	u(x,y,t) = -F(y,t) \frac{\partial}{\partial x}\ln(\varphi(x,t)),
\end{equation}
we have that the function $F(y,t)$ and $\varphi(x,t)$ should solve
diffusive equations, respectively
\begin{equation}
		\frac{\partial \varphi}{\partial t} = 	\frac{\partial^2 \varphi}{\partial x^2}
\end{equation}
and
\begin{equation}
	\frac{\partial F}{\partial t} = 	\frac{\partial^2 F}{\partial x^2}.
\end{equation}
Therefore, in view of the connection with the Hermite polynomials, we can construct a particular interesting solution also for the problem of the form 
\begin{equation}
	u(x,y,t) = -H_n(y,t)\frac{nH_{n-1}(x,t)}{H_n(x,t)}.
\end{equation}
\end{example}

In this section we have fixed the essential tools of the formalism, we will employ these methods in the forthcoming sections, aimed at combining HCt, properties of special polynomials and operational methods to derive various forms of non-linear equations, some of which not previously known.

\section{Burgers equations and special polynomials}

The polynomials defined in \eqref{H3} are sometimes referred as lacunary third order Hermite, their complete version is provided by
\begin{equation}
		H_n^{(3)}(x_1,x_2,x_3)  = n!\sum_{r=0}^{\lfloor\frac{n}{3}\rfloor}\frac{H_{n-3r}(x_1,x_2)x_3^r}{(n-3r)!r!} 
\end{equation}
satisfying the recurrences
\begin{equation}
	\frac{\partial}{\partial x_1}H_n^{(3)}(x_1,x_2,x_3) = nH_{n-1}^{(3)}(x_1,x_2,x_3)
\end{equation}
and
\begin{equation}
	\frac{\partial}{\partial x_2}H_n^{(3)}(x_1,x_2,x_3) = 	\frac{\partial^2}{\partial x_1^2}H_n^{(3)}(x_1,x_2,x_3).
\end{equation}

Moreover,
\begin{equation}
		H_n^{(3)}(x_1,0,x_3) = 	H_n^{(3)}(x_1,x_3), \qquad \quad 	H_n^{(3)}(x_1,x_2, 0) = 	H_n(x_1,x_2).
\end{equation}
According to the previously developed discussion, two different Burgers can be associated with the last two heat type equations. 

\begin{prop}
	$\forall x_1,x_2,x_3\in \mathbb{R}$, $\forall n\in\mathbb{N}$, the following PDEs hold
	\begin{equation}\label{p1}
	\frac{\partial u_n}{\partial x_3} = \frac{\partial}{\partial x_1}\left( \frac{\partial}{\partial x_1}+u_n\right)^2 u_n
	\end{equation}
	\begin{equation}\label{p2}
	\frac{\partial u_n}{\partial x_2} = \frac{\partial}{\partial x_1}\left( \frac{\partial}{\partial x_1}+u_n\right) u_n.
	\end{equation}
\end{prop}
\begin{proof}
Eqs. \eqref{p1} and \eqref{p2} are easily proved by setting
\begin{equation}
	u_n(x_1,x_2,x_3) = \frac{n H_{n-1}^{(3)}(x_1,x_2,x_3)}{H_n^{(3)}(x_1,x_2,x_3)}.
\end{equation}
\end{proof}
\begin{coro}
It is furthermore evident that in the case of the non lacunary m-th order Hermite, the following “hierarchical” system of Burgers equation holds 
\begin{equation}
	\frac{\partial u_n}{\partial x_k} = \frac{\partial}{\partial x_1}\left( \frac{\partial}{\partial x_1}+u_n\right)^{k-1} u_n, \qquad 1<k\leq m,
\end{equation}
that admits the solution
\begin{equation}
	u_n(x_1,\dots,x_m) = \frac{n H_{n-1}^{(m)}(x_1,\dots,x_m)}{H_n^{(m)}(x_1,\dots,x_m)}.
\end{equation}
\end{coro}

The examples, we have considered so far, deal with polynomials belonging to the Hermite family. The wealth of the existing polynomials can be exploited as a benchmark for the associated Burgers like equations.
To this aim we consider the two variable Laguerre polynomials.\\

We recall that the two-variable Laguerre polynomials 
\begin{equation}
	L_n(x,t) = n!\sum_{r=0}^{n}\frac{ t^{n-r} x^{r}}{(n-r)!(r!)^2}
\end{equation}
solves the Laguerre diffusive equation \cite{torre, SL}
\begin{equation}\label{ladif}
	\frac{\partial F}{\partial t} = \left(	\frac{\partial}{\partial x}x	\frac{\partial}{\partial x}\right)F.
\end{equation}
\begin{Obs}
In the recent paper \cite{logistic}, the authors have obtained the related Burgers-type equation by using a scheme similar to the one used in the previous section. Indeed,  by using the HCt 
\begin{equation} \label{cole}
	u = \frac{\partial_x F}{F}
\end{equation}
after keeping the (partial) time derivative of both sides of \eqref{cole} with respect to the variable $t$, we obtain
\begin{equation}\label{cole1}
	\frac{\partial u}{\partial t} = \frac{\displaystyle\left(\frac{\partial}{\partial t}\frac{\partial F}{\partial x}\right)F-\frac{\partial F}{\partial x}\frac{\partial F}{\partial t}}{F^2}.
\end{equation}
Being $t$ and $x$ independent variables, the associated derivatives commute, so that, also on account of Eq.\eqref{ladif}, we find
\begin{equation}
	\frac{\partial}{\partial t}\frac{\partial F}{\partial x} = \frac{\partial}{\partial x}\left(\frac{\partial}{\partial x}x\frac{\partial F}{\partial x}\right).
\end{equation}
Thus, going back to \eqref{cole1} and by using the fact that
$\partial_x F = u F$, we obtain 
\begin{equation}
	\frac{\partial u}{\partial t} = \frac{1}{F}\left(\frac{\partial}{\partial x}\left(\frac{\partial}{\partial x}(xuF)-u\frac{\partial}{\partial x}x\frac{\partial F}{\partial x} \right) \right).
\end{equation}
Working out the derivatives we eventually end up with
\begin{equation}\label{lanon}
	\frac{\partial u}{\partial t}=\frac{\partial}{\partial x}x\frac{\partial u}{\partial x}+\frac{\partial u}{\partial x}+\left(1+x\frac{\partial}{\partial x}\right) u^2
\end{equation}
which is a non-trivial non-linear reaction-diffusion type equation with known solution, once the explicit solution of the linear Laguerre-type equation \eqref{ladif} is given.
This means that we can obtain interesting exact solutions for the non-trivial nonlinear equation \eqref{lanon} by means of the Laguerre polynomials. A solution for \eqref{lanon} will be given by
\begin{equation}\label{lagct}
	u_n(x,t) = \frac{\partial_x L_n(x,t)}{L_n(x,t)}
\end{equation}
\end{Obs}

Observe that another interesting generalization of the Burgers equation is given by 
\begin{equation}\label{gio}
	\frac{\partial u}{\partial t} = \frac{\partial}{\partial x}x\frac{\partial u}{\partial x} +x	\left(\frac{\partial u}{\partial x}\right)^2,
\end{equation}
which can be linearized and reduced to the equation \eqref{ladif} by means of the transfor\-mation $u = \ln(h(x,t))$. Also in this case we can construct an explicit solution involving Laguerre polynomials to
\eqref{gio} as $u(x,t) = \ln(L_n(x,t))$.\\

Further examples of special polynomials exhibiting PDE generalizing those presented before, eventually linked to Burgers counterparts, are discussed below.  

\begin{example}
A family of polynomials of “hybrid” type is provided by (see \cite{datc})
\begin{equation}\label{2h2}
_2 \mathcal{L}_n(x,y)  = n!\sum_{r=0}^{\lfloor\frac{n}{2}\rfloor}\frac{y^r x^{n-2r}}{(n-2r)!r!^2},
\end{equation}
which has both and Hermite and Laguerre “imprinting”, for this reasons they satisfy the PDE
\begin{equation}\label{lagdi}
	\frac{\partial}{\partial y}y\frac{\partial}{\partial y}\Psi(x,y) = 
\frac{\partial^2}{\partial x^2}	\Psi(x,y),
\end{equation}
under the condition $\Psi(x,0) = x^n$. The previous equation deserves a comment. If it is indeed viewed as an initial value problem, its formal solution can be cast in the form
\begin{equation}
	\Psi(x,y)= C_0\left( y\frac{\partial^2}{\partial x^2}\right) x^n,
\end{equation}
where
\begin{equation}
	C_0(z)= \sum_{r=0}^\infty \frac{z^{r}}{r!^2}
\end{equation}
is a Bessel like function, satisfying the eigenvalue problem
\begin{equation}
	\frac{d}{dz}z\frac{d}{dz}C_0(\lambda z) = \lambda C_0(\lambda z).
\end{equation}
The search for the Burgers like form associated with polynomials of the type \eqref{2h2} can be accomplished by the use of the same means exploited for the case of non-hybrid families. The use of the transformation in \eqref{lagdi}
\begin{equation}
	u(x,y) = \ln(\Psi(x,y))
\end{equation}
 leads e.g. to
\begin{equation}
		\frac{\partial}{\partial y}y\frac{\partial u}{\partial y}+ y\left(\frac{\partial u}{\partial y}\right)^2 = \frac{\partial^2 u}{\partial x^2}+\left(\frac{\partial u}{\partial x}\right)^2.
\end{equation}
\end{example}

In this section we have seen how different forms of Burgers equations can be obtained by handling the PDE satisfied by various families of special polynomials, further comments will be given in the forthcoming concluding section.

\section{Final Comments}

In this article, we have touched different aspects regarding the impact of non-linear PDE on the theory of special polynomials. We have proved that the wealth of the relevant properties reflects itself on a comparably rich plethora of Burgers type equations linearized after a CHt involving special polynomials.
In order to provide a further flexibility of the method we consider the further example below.

\begin{example}
	We consider the PDE
\begin{equation}\label{com}
\left\lbrace \begin{split}
& \frac{\partial}{\partial y} \Psi(x,y) = \left(\alpha \frac{\partial}{\partial x}+\beta \frac{\partial^2}{\partial x^2}+\gamma \frac{\partial^3}{\partial x^3}\right)\Psi(x,y) \\[1.1ex]
& \Psi(x,0) = f(x),
\end{split}\right. 
\end{equation}
whose formal solution writes
\begin{equation}
	\Psi(x,y) = \exp\left\lbrace \displaystyle y\left(\alpha \frac{\partial}{\partial x}+\beta \frac{\partial^2}{\partial x^2}+\gamma \frac{\partial^3}{\partial x^3}\right)\right\rbrace f(x).
\end{equation}
Taking into account the generating function of the Higher order Hermite polynomials, that is \cite{dattoli}
\begin{equation}
	e^{x_1t+x_2t^2+x^3 t^3} = \sum_{n=0}^\infty \frac{t^n}{n!}H_n(x_1,x_2,x_3),
\end{equation}
the explicit solution of Eq. \eqref{com} reads
\begin{equation}
	\Psi(x,t) = \sum_{n=0}^\infty \frac{1}{n!}H_n(\alpha y,\beta y, \gamma y)\frac{\partial^n}{\partial x^n}f(x)
	\end{equation}
which for $f(x) = x^n$ reduces  eventually to \cite{dattoli}
\begin{equation}
	\Psi(x,y) = H_n(x+ \alpha y, \beta y, \gamma y). 
\end{equation}
Accordingly, the function
\begin{equation}
	u_n(x,y) = \frac{n H_{n-1}(x+ \alpha y, \beta y, \gamma y)}{H_{n}(x+ \alpha y, \beta y, \gamma y)}
\end{equation}
is one of the solutions of the non-linear Burgers associated to the Eq. \eqref{com}
\begin{equation}
\frac{\partial u_n}{\partial y}	= \frac{\partial}{\partial x}\bigg[\alpha + \beta\left(\frac{\partial}{\partial x}+u_n\right)+\gamma \left(\frac{\partial}{\partial x}+u_n\right)^2\bigg]u_n,
\end{equation}
which is also a consequence of the previously quoted “hierarchical” properties.
\end{example}

Before closing the article, we would like to comment on the role of the functions
\begin{equation}\label{Phinm}
	\Phi_n^{(m)}(x,y) = \frac{nH_{n-1}^{(m)}(x,y)}{H_{n}^{(m)}(x,y)}
\end{equation}
which play a role less trivial than it might appear at first glance. They represent indeed a general form of rational solutions of Burgers like equations.  The interest for these form of solutions is motivated by the fact that the non-linear equations we have dealt with are simple analogue of the Navier-Stokes equations, and can be exploited as test for the relevant numerical solutions, see e.g. \cite{guz}. Even though the problem deserves a more appropriate treatment, we will afford in a forthcoming paper, we note that, for example
\begin{align}
\nonumber &	\Phi_2^{(2)}(x,y) = \frac{2x}{2y+x^2},\\
\nonumber &\Phi_4^{(2)}(x,y) = \frac{4x^3+24xy}{x^4+12x^2y+12y^2},\\
\nonumber &\Phi_4^{(4)}(x,y) = \frac{6y+x^3}{6xy+\frac{x^4}{4}},\\
\nonumber & \Phi_9^{(4)}(x,y) = \frac{9.072\cdot 10^4 x^2 y^2 + 3.024\cdot 10^3 x^5 + 9 x^8}{6.048\cdot 10^4 y^3 + 3.024\cdot 10^4 x^3 y^2 + 504 x^6 + x^9}
\end{align}
are rational solutions of the Burgers equation \eqref{bu} which can be compared with analogous forms, obtained e.g. in \cite{av}, \cite{kud1}, \cite{kud} and \cite{zuparic}.
The behavior of these functions is given in Figs. \ref{Fig1} where we have reported the plots of $\Phi_n^{(m)}(x,.)$ vs. x , at fixed y for different values of the index $n$ and order $m$. The relevant behavior is that exhibited by the solutions reported in Ref. \cite{av}, where the problem of deriving  new analytical solutions of rational type has been addressed. A more exhaustive 3-D view is given in Figs. \ref{Fig2}.\\

\begin{figure}[h]
	\centering
	\begin{subfigure}[c]{0.48\textwidth}
		\centering
		\includegraphics[width=0.8\linewidth]{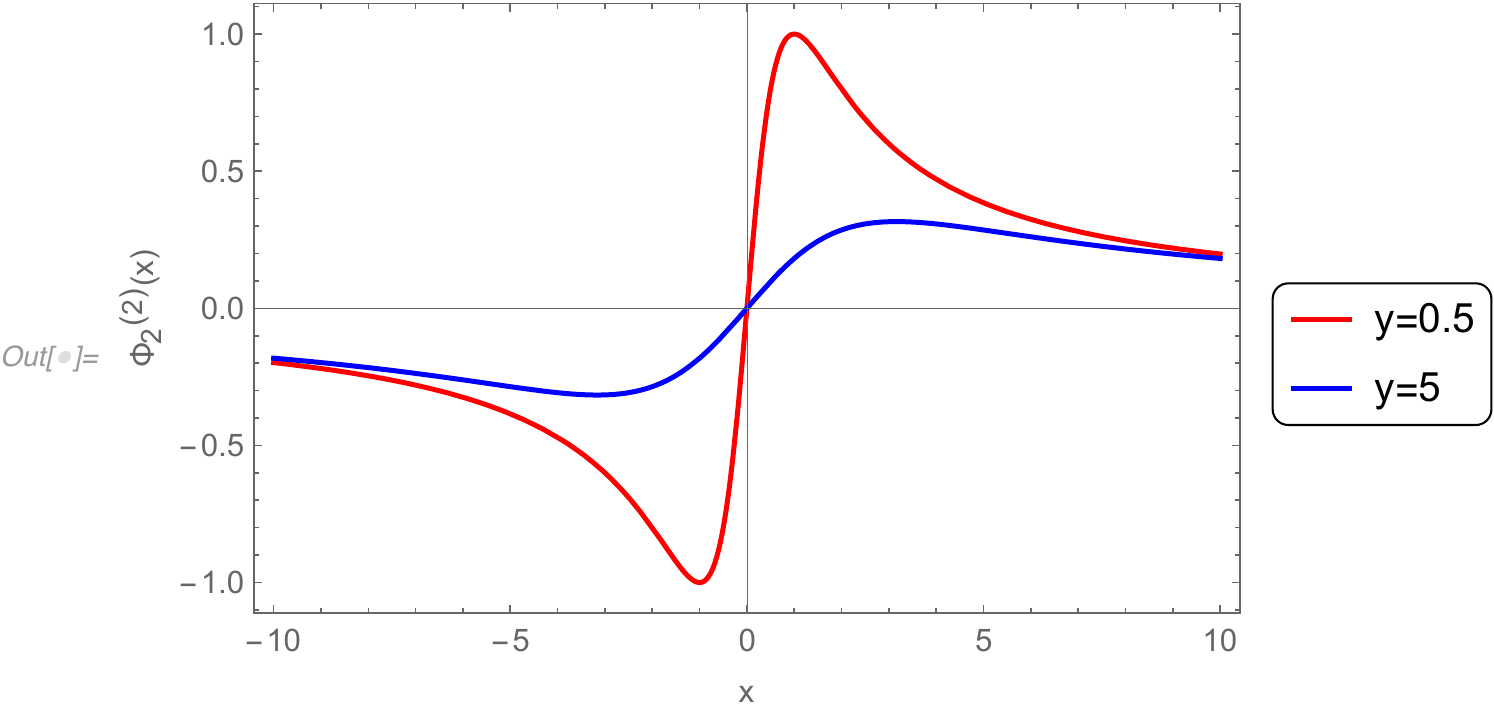}
		\caption{$n=m=2$}
	\end{subfigure}
	\begin{subfigure}[c]{0.48\textwidth}
		\centering
		\includegraphics[width=0.8\linewidth]{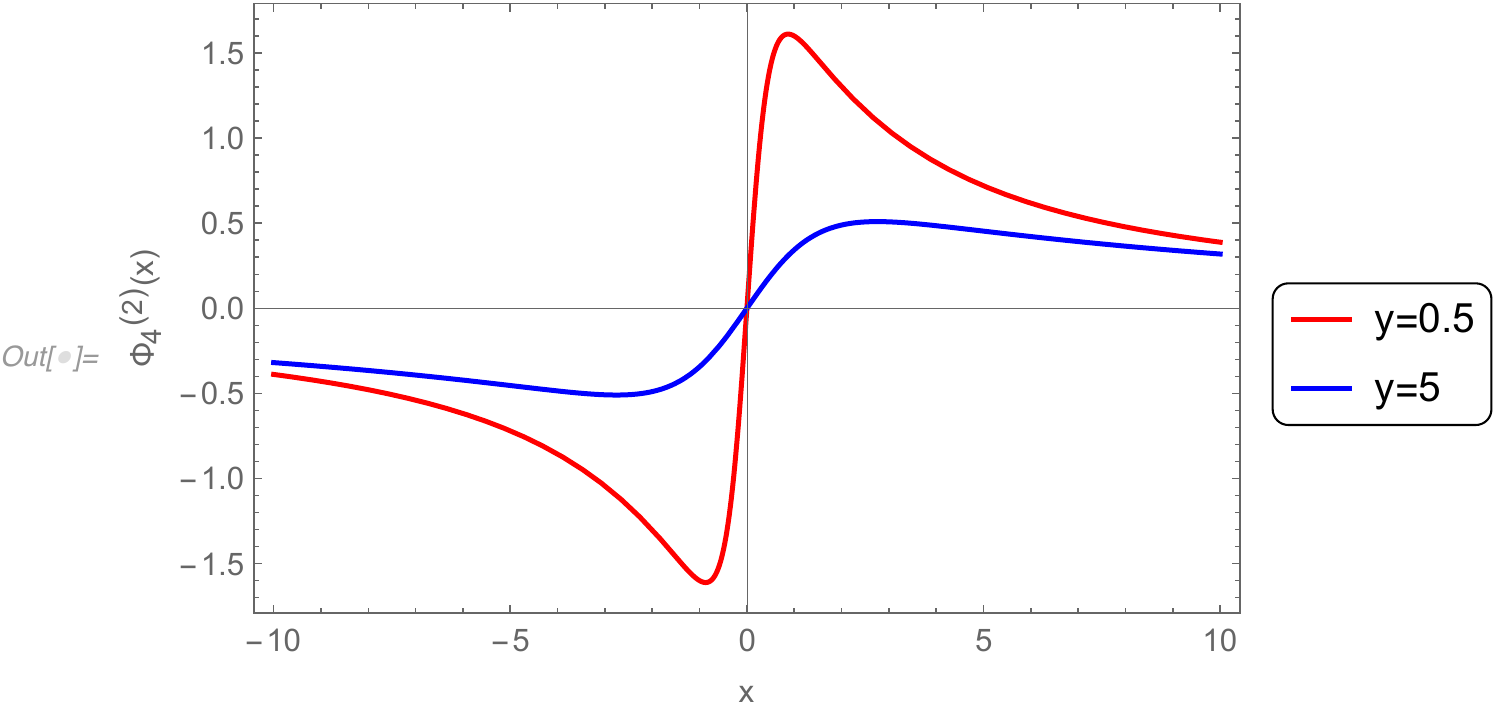}
		\caption{$n=4,m=2$}
	\end{subfigure}
	\begin{subfigure}[c]{0.48\textwidth}
	\centering
	\includegraphics[width=0.8\linewidth]{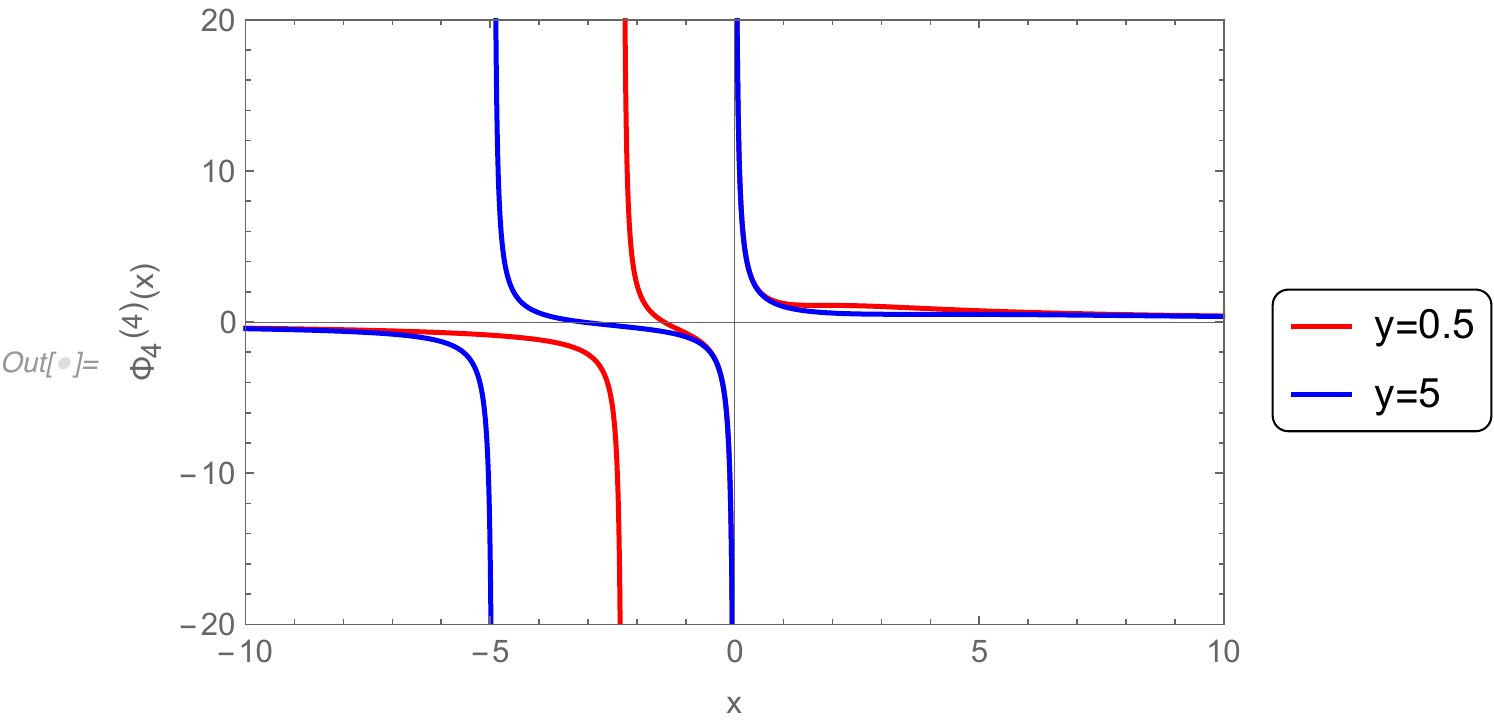}
	\caption{$n=m=4$}
\end{subfigure}
	\begin{subfigure}[c]{0.48\textwidth}
	\centering
	\includegraphics[width=0.8\linewidth]{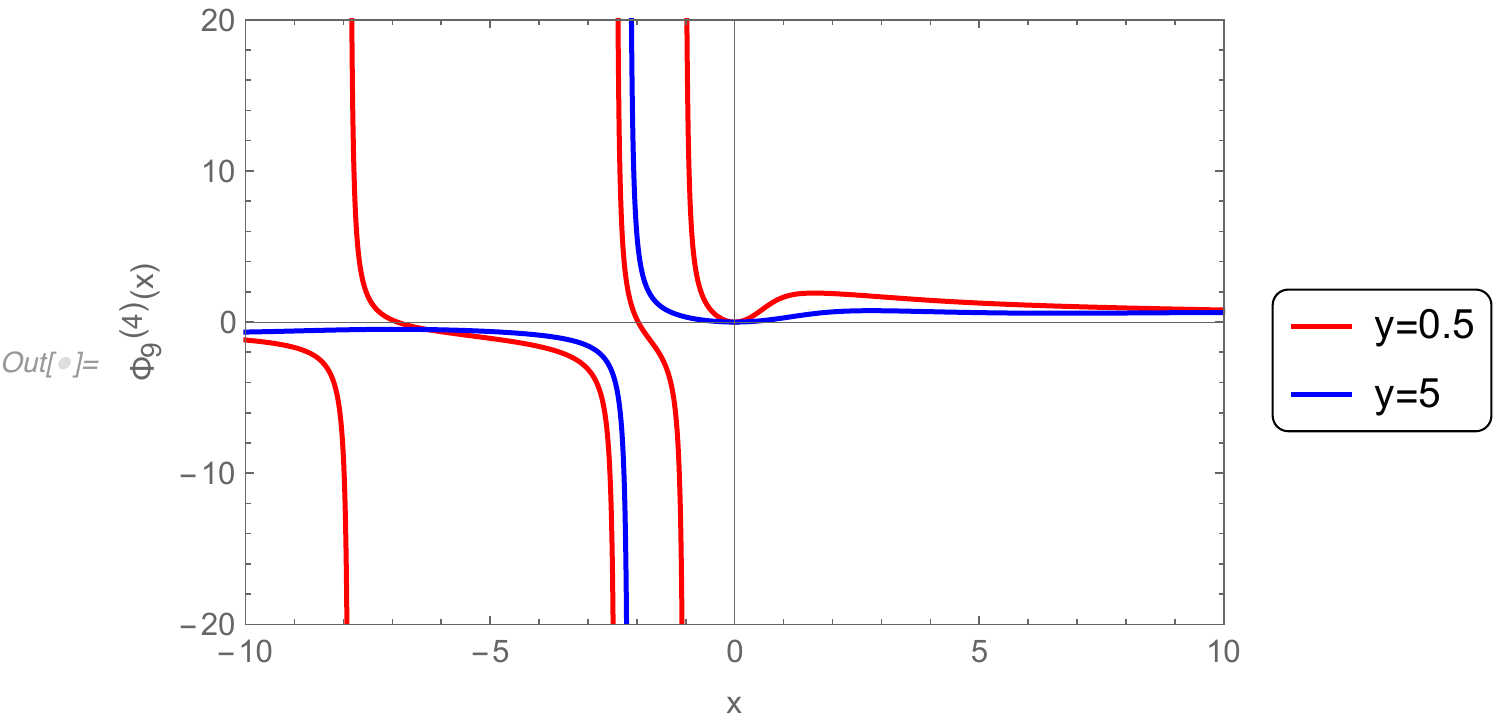}
	\caption{$n=9,m=4$}
\end{subfigure}
	\caption{Comparison of $\Phi_n^{(m)}(x,y)$ vs. $x$ for fixed $y$ values and different $n$ and $m$ values.}\label{Fig1} 
\end{figure}

\begin{figure}[h]
	\centering
	\begin{subfigure}[c]{0.4\textwidth}
		\centering
		\includegraphics[width=0.7\linewidth]{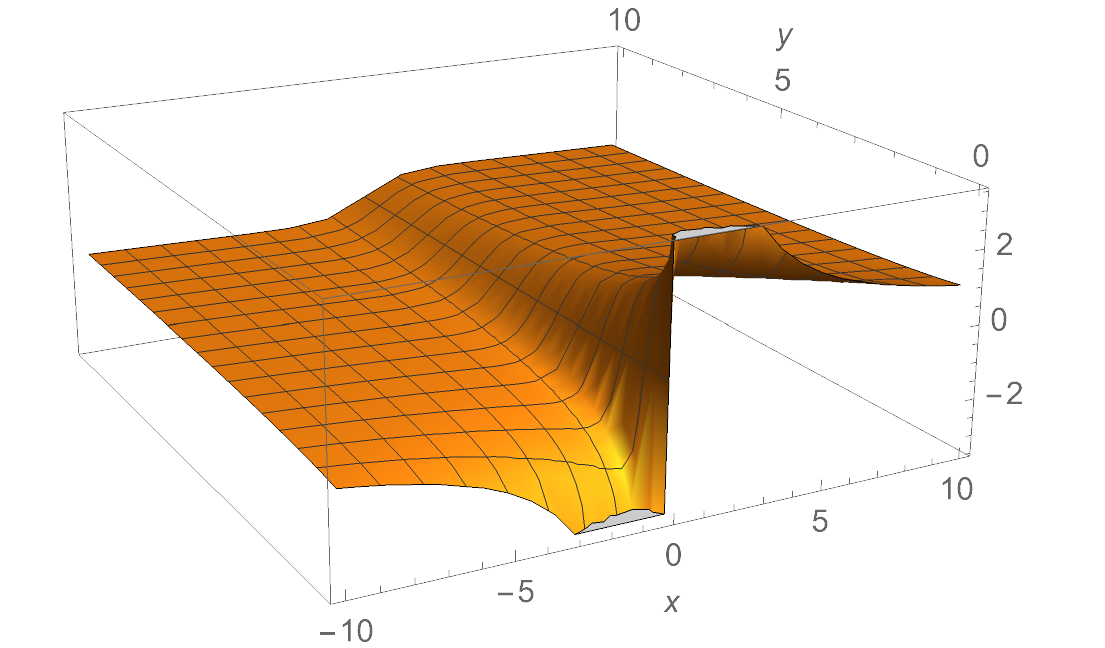}
		\caption{$n=10,m=2$}
	\end{subfigure}
	\begin{subfigure}[c]{0.4\textwidth}
	\centering
	\includegraphics[width=0.7\linewidth]{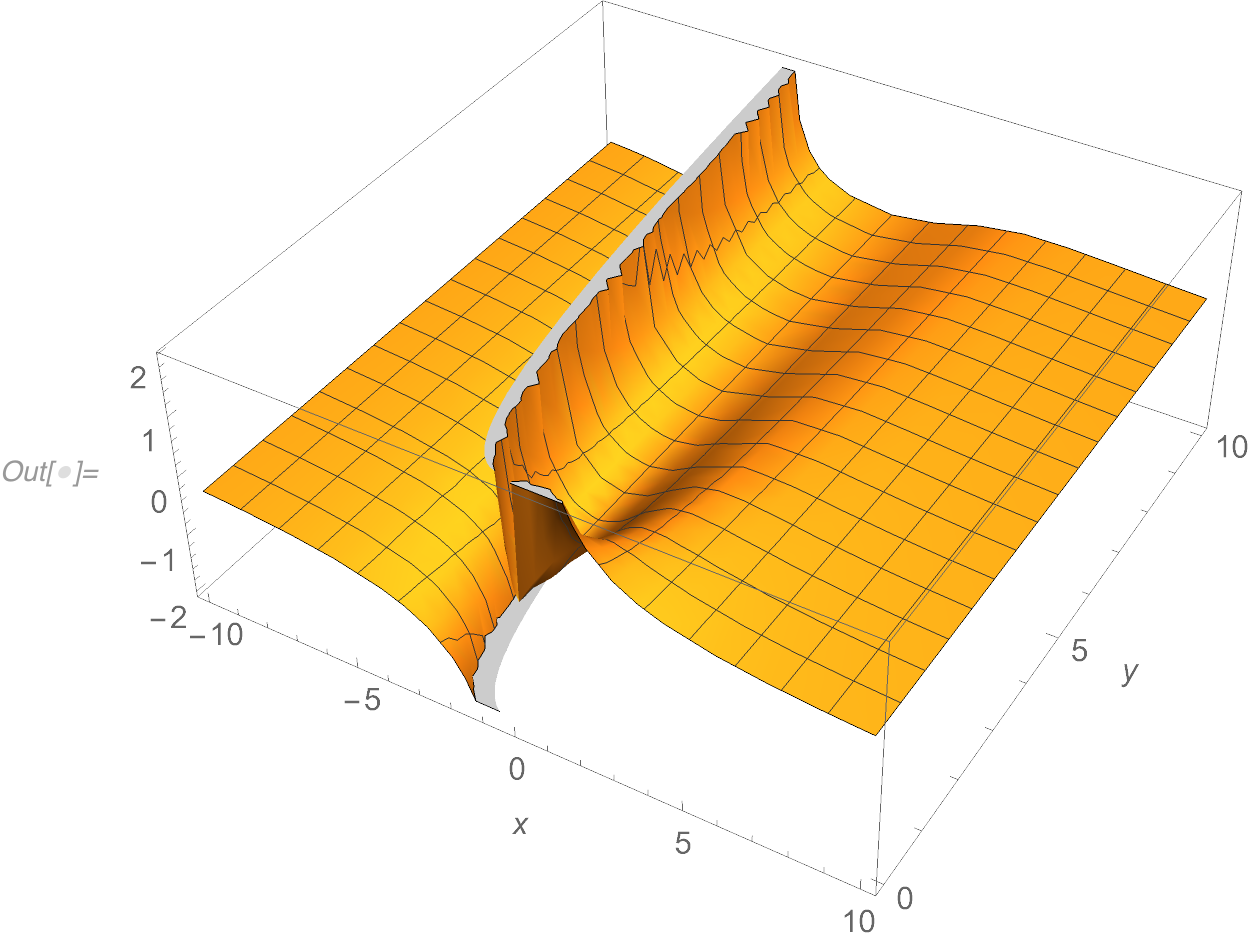}
	\caption{$n=m=3$}
\end{subfigure}
	\begin{subfigure}[c]{0.4\textwidth}
		\centering
		\includegraphics[width=0.7\linewidth]{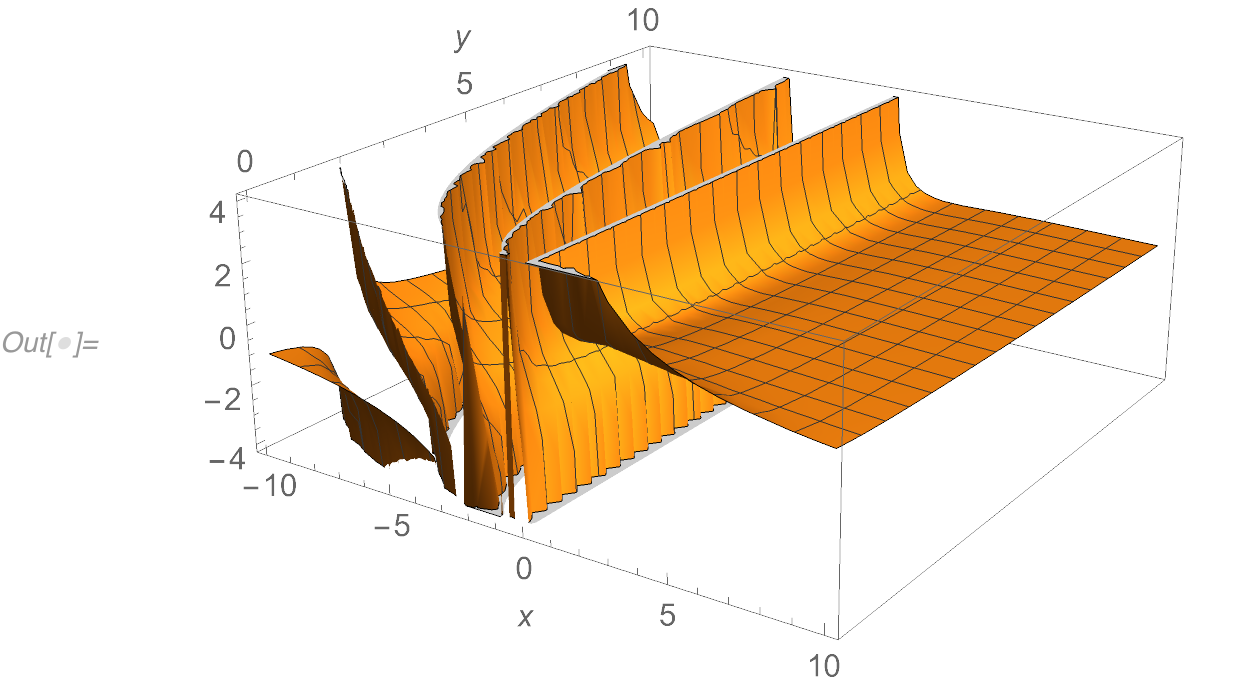}
		\caption{$n=10,m=3$}
	\end{subfigure}
	\begin{subfigure}[c]{0.4\textwidth}
	\centering
	\includegraphics[width=0.7\linewidth]{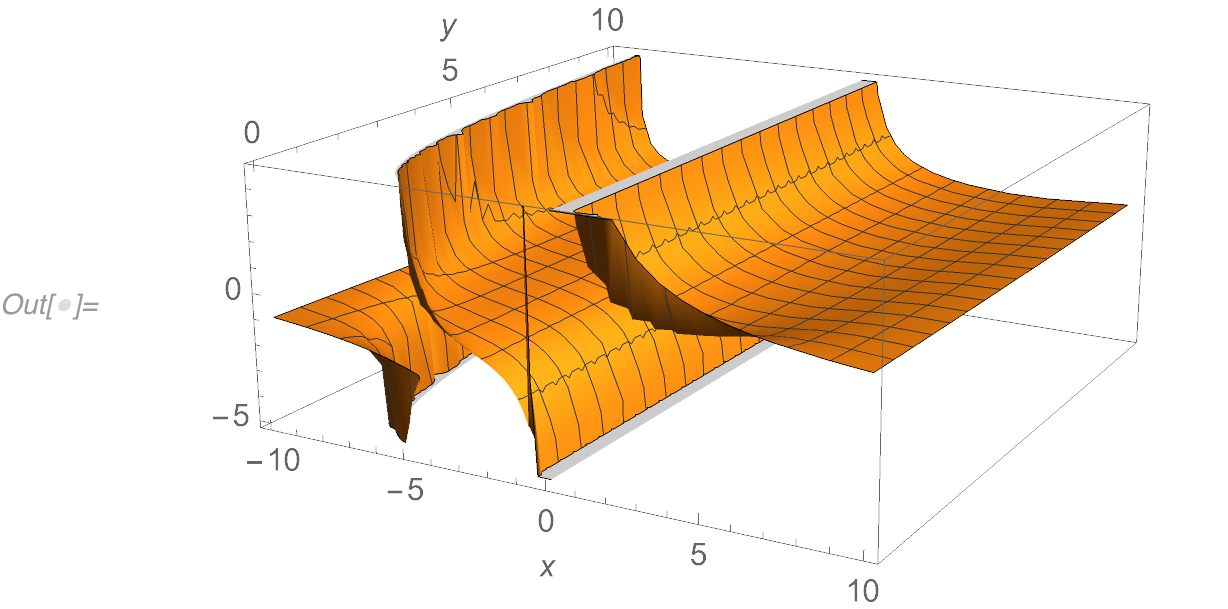}
	\caption{$n=10,m=7$}
\end{subfigure}
	\caption{Eq. \eqref{Phinm} for different $n$ and $m$ values.}\label{Fig2} 
\end{figure}

\begin{os}
The same considerations apply to the case of the two variable Laguerre. Examples of rational solutions of the Burgers type associated with the Laguerre diffusion equation (given in Eq. \eqref{lagct}) are listed below for $n = 3$ and $n = 7$

\begin{equation}\label{u3f}
u_3(x,y) = \frac{\frac{x^2}{12}+\frac{xy}{2}+\frac{y^2}{2}}{\frac{x^3}{36}+\frac{x^2y}{4}+\frac{xy^2}{2}+\frac{y^3}{6}}
\end{equation}

\begin{equation}\label{u7f}
	u_7(x,y) = \frac{\frac{x^6}{3628800}+\frac{x^5y}{86400}+
\frac{x^4y^2}{5760}+\frac{x^3y^3}{864}+\frac{x^2y^4}{288}+\frac{xy^5}{240}+\frac{y^6}{720}}{\frac{x^7}{25401600}+\frac{x^6y}{518400}+\frac{x^5y^2}{28800}+\frac{x^4y^3}{3456}+\frac{x^3y^4}{864}+\frac{x^2y^5}{480}+\frac{xy^6}{720}+\frac{y^7}{5040}}.
\end{equation}
The relevant plots are given in Figs. \ref{Fig3}. The structure is not dissimilar from that exhibited by the Hermite counterparts.
More appropriate comments will be provided elsewhere.

\begin{figure}[h]
	\centering
	\begin{subfigure}[c]{0.48\textwidth}
		\centering
		\includegraphics[width=0.8\linewidth]{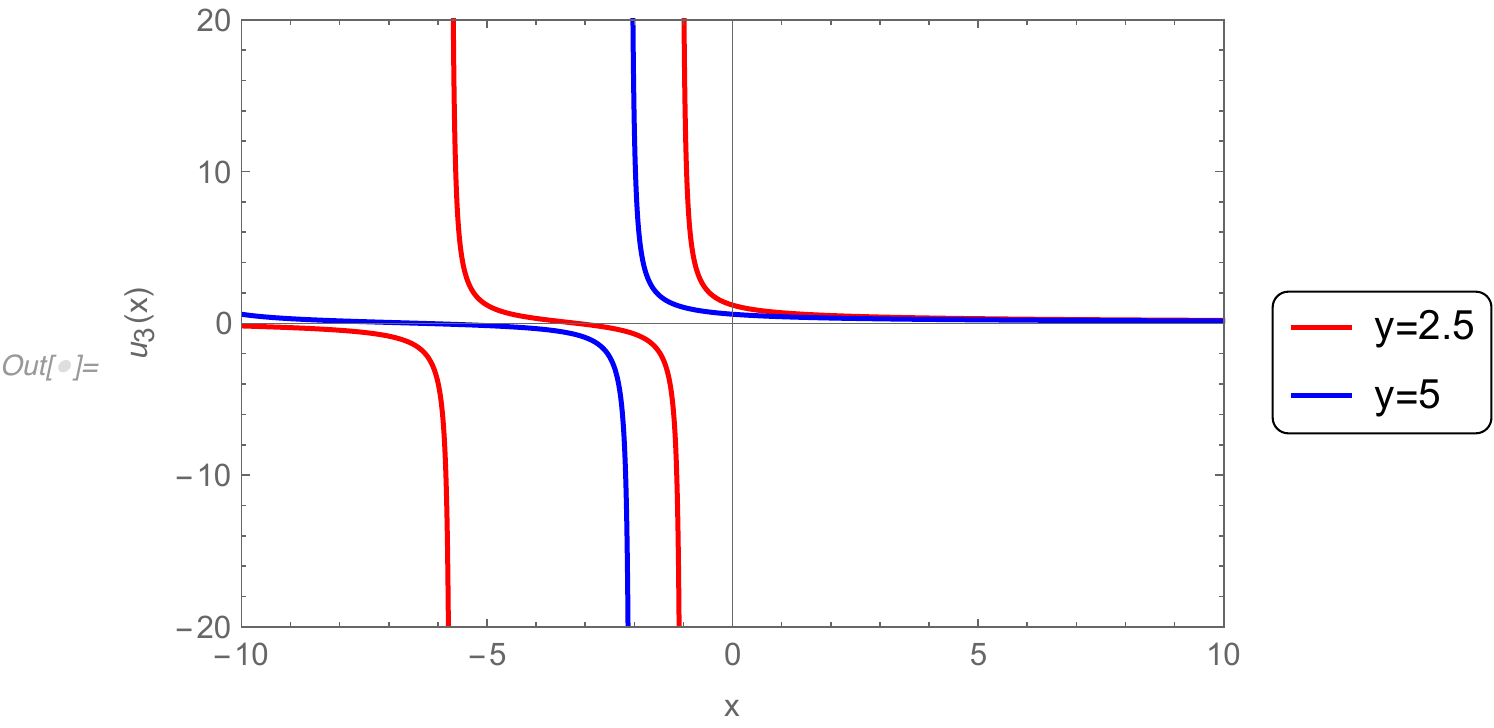}
		\caption{$n=3$}
	\end{subfigure}
	\begin{subfigure}[c]{0.48\textwidth}
		\centering
		\includegraphics[width=0.8\linewidth]{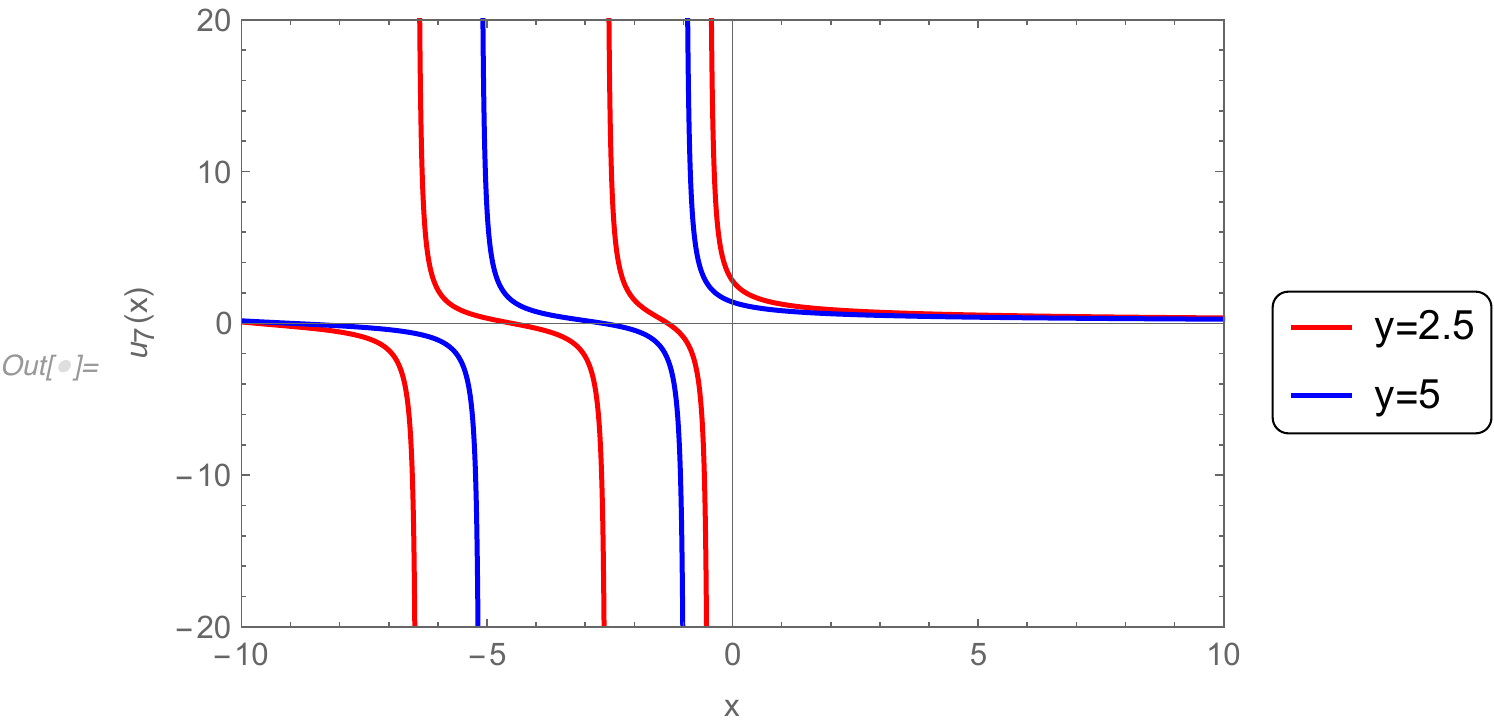}
		\caption{$n=7$}
	\end{subfigure}
	\caption{Eqs. \eqref{u3f}-\eqref{u7f} vs. $x$ for fixed $y$ values}\label{Fig3} 
\end{figure}
\end{os}

We have so far discussed the use of Hermite polynomials to construct exact solutions of Burgers  equations. This point of view can be reversed. The wealth of the properties of the  non-linear transformation we have dealt with can be exploited to obtain previously unknown properties between Hermite polynomials, like e. g.
\begin{equation}
	\left(\frac{\partial}{\partial x}+F_n(x,y)\right)F_n(x,y) = S_n(x,y),
\end{equation}
where
$$F_n(x,y) = \frac{(n-1)H_{n-2}(x,y)}{H_{n-1}(x,y)}$$
and
$$S_n(x,y) = \frac{(n-1)(n-2)H_{n-3}(x,y)}{H_{n-1}(x,y)},$$
which is a particular case of more general identities obtainable  from the handling of the Burgers auxiliary functions.
These problems and the construction of new exact solutions of the Burgers equation in terms of Hermite polynomials will be carefully studied in a forthcoming dedicated paper.

\end{document}